\documentclass[12pt, reqno]{amsproc}
\usepackage{amsmath, amsthm, amscd, amsfonts, amssymb, graphicx}
\usepackage{geometry}
\usepackage{cite}
\usepackage{calc}
\usepackage{color,hyperref}
\definecolor{darkblue}{rgb}{0.0,0.0,0.3}
\hypersetup{colorlinks,breaklinks,
            linkcolor=red,urlcolor=pink,
            anchorcolor=darkblue,citecolor=blue}
\newcommand{\makeheading}[2][]%
        {\hspace*{-\marginparsep minus \marginparwidth}%
         \begin{minipage}[t]{\textwidth+\marginparwidth+\marginparsep}%
             {\large \bfseries #2 \hfill #1}\\[-0.15\baselineskip]%
                 \rule{\columnwidth}{1pt}%
         \end{minipage}}

\newtheorem{theorem}{\bf{Theorem}}[section] 
\newtheorem{lemma}[theorem]{\bf{Lemma}}     

\theoremstyle{definition}

\title[Riemann-Stieltjes operator]{Riemann-Stieltjes operators on some  general spaces on the unit ball and their essential norm}

\author{\sc\bf M. S. Al Ghafri, M. Hassanlou, Y. Estaremi and  Z. Huang  }
\address{ \sc M. S. Al Ghafri}
\email{mo86said@gmail.com}
\address{Department of Mathematics, Sultan Qaboos University, Muscat, Oman}
\address{ \sc M. Hassanlou}
\email{m.hassanlou@urmia.ac.ir}
\address{Engineering Faculty of Khoy, Urmia University of Technology, Urmia, Iran}
\address{\sc Y. Estaremi}
\email{y.estaremi@gu.ac.ir}
\address{Department of Mathematics, Faculty of Sciences, Golestan University, Gorgan, Iran.}

\keywords{Riemann-Stieltjes operator,  Essential norm, Zygmund type space.\\
\indent $^{*}$ Corresponding author}
\begin{document}

\vspace{1cm} \setcounter{page}{1} \thispagestyle{empty}
\maketitle
\begin{abstract}
The objective of this paper is to find an estimation of the essential norm of Riemann-Stieltjes operator from the spaces $F(p,q,s)$ ( general function space) and $H(p,q,\phi)$ (mixed-norm space)  into Zygmund-type space in the unit ball of $\mathbb{C}^n$.
\end{abstract}
\section{Introduction and Preliminaries}
Let $\mathbb{B}_n$ be the unit ball in $\mathbb{C}^n$ and $H(\mathbb{B}_n)$ be the class of all holomorphic functions on $\mathbb{B}_n$. For $f \in H(\mathbb{B}_n)$, with the Taylor expansion $f(z) = \sum_{|\beta| \geq 0} a_{\beta} z^{\beta}$, let
$\mathcal{R} f(z) = \sum_{|\beta| \geq 0} |\beta| a_{\beta} z^{\beta}$ be the radial derivative of $f$. Here $z = (z_1, z_2, \cdots, z_n) \in \mathbb{B}_n$, $\beta = (\beta_1, \beta_2, \cdots, \beta_n)$ is a multi-index, $z^{\beta} = z_1^{\beta_1} \cdots z_n^{\beta_n}$ and $|\beta| = \beta_1 + \beta_2 + \cdots + \beta_n$. It is easy to see that

$$\mathcal{R} f(z) = \sum_{j=1}^{n} z_j \frac{\partial f}{\partial z_j} (z).$$

For more information one can see \cite{zhu1}. The iterates of $\mathcal{R}$ is defined by $\mathcal{R}^m f (z) = \mathcal{R} (\mathcal{R}^{m-1} f (z))$, for $m \in \mathbb{N} \backslash \{ 1 \}$.

Let $0< p< \infty$, $0 \leq s < \infty$ and $-n-1 < q < \infty$. The space $F(p,q,s)$ is the space of all analytic functions in $\mathbb{B}_n$ for which
$$ \| f \|_{F(p,q,s)}^p = |f(0)|^p + \sup_{a \in \mathbb{B}_n} \int_{\mathbb{B}_n} |\mathcal{R}f(z)|^p (1-|z|^2)^q g^{s} (z,a) dV(z) < \infty. $$
Here $dV$ is normalized area measure on $\mathbb{B}_n$, $g(z,a) = \log | \frac{1}{\varphi_a (z)}|$, $\varphi_a (z) = \frac{a-z}{1-\overline{a}z}$.
Using complex gradient $\nabla$ and complex invariant gradient $\widetilde{\nabla}$, there are some equivalent noms for the functions in this space, see \cite{zhang} for more details. The above space, introduced in \cite{zhao1}, is called general function space which include some classes of functions like Bergman, Bloch, Hardy, BMOA and $Q_p$ spaces. See \cite{zhao1} for more details.

Beside the general space $F(p,q,s)$, we use another general class of analytic functions $H(p,q,\phi)$ which is called mixed-norm space.\\

The function $\phi$ on $[0,1)$ is called normal if it is a positive continuous function for which there are two constants $b > a > 0$ such that
\begin{itemize}
  \item [(i)] $\frac{\phi(r)}{(1-r)^a}$ is non-increasing in $[0,1)$ and
  $\frac{\phi(r)}{(1-r)^a} \downarrow 0$
  \item [(ii)] $\frac{\phi(r)}{(1-r)^b}$ is non-decreasing  in $[0,1)$ and
  $\frac{\phi(r)}{(1-r)^b} \uparrow \infty$
\end{itemize}
as $r \rightarrow 1^-$. By a normal function $\phi: \mathbb{B}_n \rightarrow [0,\infty )$ we mean that $\phi(|z|)$ is normal and also $\phi$ is radial, $\phi(z) = \phi(|z|)$.

 For $0 < p,q< \infty$ and normal function $\phi$ on $[0,1)$, the mixed-norm space $H(p,q,\phi)$ consists of all analytic functions in $\mathbb{B}_n$ for which
$$ \| f \|_{p,q, \phi}^p = \int_0^1 M_q^p (f,r) \frac{\phi^p (r)}{1-r} dr < \infty $$
where
$$ M_q (f,r) =\left ( \frac{1}{2\pi} \int_0^{2\pi} |f(r e^{i\theta})|^q d \theta \right)^{1/q}. $$
If $1\leq p < \infty$, then  $H(p,q,\phi)$ becomes a Banach space equipped with the norm
$\| .\|_{p,q,\phi}$. In the case $0 < p<1$, $\| .\|_{p,q,\phi}$ is a quasi-norm on
$H(p,q,\phi)$ and this space is a Fr$\acute{e}$chet space but is not a Banach space.
For $\alpha>-1$, $\phi(r) = (1-r)^{(\alpha+1)/p}$ is a normal function. If $p=q$, then $H(p,p,(1-r)^{(\alpha+1)/p})$ is weighted Bergman space $A_{\alpha}^p$ which is defined by
$$ A_{\alpha}^p = \{ f \in H(\mathbb{B}_n): \| f \|_{A_{\alpha}^p}^p = \int_{\mathbb{B}_n} |f(z)|^p (1-|z|^2)^{\alpha} dV(z) < \infty \}.$$

Let $\mu$ be a positive continuous function on [0, 1), such a function is called weight.  The Bloch-type space $\mathcal{B}_{\mu}$ and Zygmund-type space $\mathcal{Z}_{\mu}$ are defined as
\begin{align*}
\mathcal{B}_{\mu} = & \{ f \in H(\mathbb{B}_n): \| f \|_{\mathcal{B}_{\mu}} = |f(0)| + \sup_{z \in \mathbb{B}_n} \mu(|z|) |\mathcal{R}f(z)| < \infty \} \\
\mathcal{Z}_{\mu} = & \{ f \in  H(\mathbb{B}_n) : \| f \|_{\mathcal{Z}_{\mu}} = |f(0)| +  \sup_{z \in \mathbb{B}_n} \mu(|z|) |\mathcal{R}^2f(z)|< \infty \}.
\end{align*}
When $\mu(r) = (1-r^2)$ then we have classical Bloch and Zygmund space, $\mathcal{B}$ and $\mathcal{Z}$. If $\mu(r) = (1-r^2)^{\alpha}$, $\alpha>0$, then
$\alpha$-Bloch and Zygmund spaces are obtained,
$ \mathcal{B}^{\alpha} $ and $\mathcal{Z}^{\alpha}$.
Let $g \in H(\mathbb{B}_n)$. The Riemann-Stieltjes operator $L_g$ is defined by
$$ L_g f (z) = \int_0^1 \mathcal{R} f(tz) g(tz) \frac{dt}{t}, \ \ f \in H(\mathbb{B}_n), \ \ z \in \mathbb{B}_n.$$
Bounded and compact Riemann-Stieltjes operator from general space of analytic functions to Zygmund-type spaces and little Zygmund-type spaces on the unit ball are characterized in \cite{liu1}. Liu and Yu have investigated boundedness and compactness of $L_g$ from mixed-norm spaces to Zygmund-type spaces in \cite{liu2}. Essential norm of Riemann-Stieltjes operator on
weighted Bergman spaces with doubling weights is estimated in \cite{hu}. Volterra-type operators from $F(p,q,s)$ space to Bloch-Orlicz and Zygmund-Orlicz spaces are investigated in
\cite{liang2}. For more results on integral operators or generalized integral operators on some spaces of analytic functions one can see \cite{li1,li2,li3,liang1, liang, mi,pau, peng,st2, yang, zhao1, zhang1, zhu2}
and references therein.

Motivated by \cite{liu1,liu2}, we are going to find an estimation of the essential norm of the Riemann-Stieltjes operator from the general space of analytic functions, $F(p,q,s)$, and mixed-norm space, $H(p,q,\phi)$, into Zygmund-type spaces in the unit ball.

In this paper, for real scalars $A$ and $B$, the notation $\preceq$
means $A \leq C B$ for some positive constant $C$. Also, the notation
$A \approx B$ means $A \preceq B$ and $B \preceq A$. The essential norm of an operator $T$, $\| T \|_e$, is the distance
of $T$ from the space of all compact operators.

\section{Main results}
In this section, first we recall some basic lemmas from the literature for later use in the paper. Then we estimate the essential norm of Riemann-Stieltjes operator $L_g$ from the general function space and mixed-norm space into Zygmund-type space in the unit ball of $\mathbb{C}^n$.

\begin{lemma} \cite{st1}
  For every $f ,g \in H(\mathbb{B}_n)$,
  $$ \mathcal{R} L_g (f)(z) = \mathcal{R} f(z) g(z)\ \ \ \ \ \ z\in \mathbb{B}_n. $$
\end{lemma}
\begin{lemma} \label{l24} \cite{zhang1}
  Let $0 < p,s <\infty$, $\max \{ -n-1, -s-1 \} < q < \infty$ and
  $f \in F(p,q,s)$. Then $f \in \mathcal{B}^{\frac{n+1+q}{p}}$ and
  $$\| f \|_{\mathcal{B}^{\frac{n+1+q}{p}}} \leq C \| f \|_{F(p,q,s)}.$$
\end{lemma}
According to the Proposition 2.1 of \cite{Lc}, for $\alpha >0$ and $m \in \mathbb{N}$, we have $f \in \mathcal{B}^{\alpha}$ if and only if
\begin{equation}\label{r20}
  \sup_{z \in \mathbb{B}_n} |\mathcal{R}^m f(z)| (1-|z|^2)^{\alpha + m -1} < \infty.
\end{equation}
\begin{lemma} \label{l20} \cite{st2}
Assume that $m \in \mathbb{N}$, $p,q \in (0,\infty)$, $\phi$ is normal and $f \in H(p,q,\phi)$. Then there exists a positive constant $C$ independent of $f$ such that
\begin{equation*}
  |\mathcal{R}^m f(z)| \leq  \frac{C |z|}{\phi(|z|) (1-|z|^2)^{m + n/q }}\| f \|_{p,q,\phi}, \ \ \ z \in \mathbb{B}_n.
\end{equation*}
\end{lemma}
\begin{lemma}  \label{l21} \cite{li1}
 Let $0 < p,s <\infty$, $\max \{ -n-1, -s-1 \} < q < \infty$, $\mu$ be a weight and $g \in H(\mathbb{B}_n)$.  Then
$L_g :F(p,q,s) \rightarrow \mathcal{Z}_{\mu}$ is compact if and only if $L_g :F(p,q,s) \rightarrow \mathcal{Z}_{\mu}$ is bounded and
for any bounded sequence $\{ f_k \}$ in $F(p,q,s)$, which converges uniformly to zero on compact subsets of $\mathbb{B}_n$, as $k \rightarrow \infty$,  we have
$\| L_g f_k \|_{\mathcal{Z}_{\mu}} \rightarrow 0 $, as $k \rightarrow \infty$.
\end{lemma}
\begin{lemma}  \label{l22} \cite{li2}
 Let $0 < p,q <\infty$,  $\mu$ be a weight, $\phi$ be a normal function and $g \in H(\mathbb{B}_n)$. Then
$L_g :H(p,q,\phi) \rightarrow \mathcal{Z}_{\mu}$ is compact if and only if $L_g :H(p,q,\phi) \rightarrow \mathcal{Z}_{\mu}$ is bounded and for any bounded sequence $\{ f_k \}$ in $H(p,q,\phi)$ which converges uniformly to zero on compact subsets of $\mathbb{B}_n$ as $k \rightarrow \infty$,  we have
$\| L_g f_k \|_{\mathcal{Z}_{\mu}} \rightarrow 0 $ as $k \rightarrow \infty$.
\end{lemma}
By the above mentioned lemmas, in the next theorem we find an estimation of the essential norm of the Riemann-Stieltjes operator $L_g : F(p,q,s) \rightarrow \mathcal{Z}_{\mu}$. For simplicity, put $\alpha = \frac{n+1 + q}{p}$.
\begin{theorem} \label{th2}
Let $0 < p,s <\infty$, $\max \{ -n-1, -s-1 \} < q < \infty$, $\mu$ be a weight and $g \in H(\mathbb{B}_n)$.  If $L_g :F(p,q,s) \rightarrow \mathcal{Z}_{\mu}$ is bounded, then
\begin{equation*}
  \| L_g \|_{e,F(p,q,s) \rightarrow \mathcal{Z}_{\mu} } \approx \max
 \left \{ \limsup_{|z| \rightarrow 1} \frac{\mu(|z|) |g(z)|}{(1-|z|^2)^{\alpha +1}}, \ \limsup_{|z| \rightarrow 1} \frac{\mu(|z|) |\mathcal{R}g(z)|}{(1-|z|^2)^{\alpha}} \right \}.
\end{equation*}
\end{theorem}
\begin{proof}
Since $L_g$ is bounded, then there exists $C >0$ such that

$$ \| L_g f \|_{\mathcal{Z}_{\mu}} \leq C \| f \|_{F(p,q,s)}, \ \ \ \ \ \ \text{for all} \ \ \ f \in F(p,q,s).$$

By applying this inequality to $f_1 (z) =z$, we obtain
\begin{equation}\label{r15}
 \sup_{z \in \mathbb{B}_n} \mu (|z|) | \mathcal{R} g(z)|<\infty.
\end{equation}
Also, by applying the inequality to $f_2 (z) = z^2$, we have
\begin{equation}\label{r16}
 \sup_{z \in \mathbb{B}_n} \mu (|z|) |  g(z)|<\infty.
\end{equation}
Let $\{ z_i \}_{i \in \mathbb{N}}$ be a sequence in $\mathbb{D}$, such that $|z_i| \rightarrow 1$. We define the sequence of functions as follows

$$f_i (z) = \frac{(1-|z_i|^2)^{\alpha +1}}{(1-\overline{z_i}z)^{2\alpha}}  - 2 \frac{1-|z_i|^2}{(1-\overline{z_i}z)^{\alpha}}, \ \ \ \ \ i\in \mathbb{N}.$$

Direct calculations shows that for every $i\in \mathbb{N}$, $f_i \in F(p,q,s)$, $\sup_{i} \| f_i \|_{F(p,q,s)} \leq C$ and $f_i \rightarrow 0$, uniformly on compact subsets of $\mathbb{B}_n$. Also,
$$\mathcal{R} f_i (z_i) =0\ \ \ \ \text{and} \ \ \ \  \mathcal{R}^2 f_i (z_i)= 2 \alpha^2 \frac{|z_i|^4}{(1-|z_i|^2)^{\alpha +1}}.$$
So for any compact operator $K: F(p,q,s) \rightarrow \mathcal{Z}_{\mu}$ we have

$$\lim_{i \rightarrow \infty} \| K f_i \|_{\mathcal{Z}_{\mu}} =0.$$

By the above observations we have
\begin{align*}
  \|  L_g - K \|_{F(p,q,s) \rightarrow \mathcal{Z}_{\mu}} &\succeq
  \limsup_{i \rightarrow \infty} \|  L_g f_i - K f_i \|_{\mathcal{Z}_{\mu}} \\
  & \geq
 \limsup_{i \rightarrow \infty}\| L_g f_i \|_{\mathcal{Z}_{\mu}}-
 \limsup_{i \rightarrow \infty}\| K f_i \|_{\mathcal{Z}_{\mu}}\nonumber\\
 & = \limsup_{i \rightarrow \infty} \mu(|z_i|) |\mathcal{R}^2 ( L_g f_i ) (z_j)| \nonumber \\
 &\succeq  \limsup_{i \rightarrow \infty} \frac{\mu(|z_i|) | g(z_i)| |z_i|^4}{ (1-|z_i|^2)^{\alpha +1} }.
 \end{align*}
Therefore,
\begin{align} \label{r12}
\| L_g \|_{e, F(p,q,s) \rightarrow \mathcal{Z}_{\mu}}= \inf_{K}\| L_g - K \|_{F(p,q,s) \rightarrow \mathcal{Z}_{\mu}}  \succeq
  \limsup_{|z| \rightarrow 1} \frac{\mu(|z|) | g(z)| }{  (1-|z|^2)^{\alpha +1} }.
 \end{align}
Now we define the sequence $\{ h_i \}$ as follows
$$ h_i (z) = \frac{(1-|z_i|^2)^{\alpha +1}}{(1-\overline{z_i}z)^{2\alpha}}  - \frac{2 + 4 \alpha}{\alpha+1} \frac{1-|z_i|^2}{(1-\overline{z_i}z)^{\alpha}}.  $$
Then $\{ h_j \}$ is a bounded sequence in $F(p,q,s)$ which converges to 0 uniformly on compact subsets of $\mathbb{B}_n$. Also
$$ \mathcal{R} h_i (z_i) =\mathcal{R}^2 h_i (z_i) = - \frac{2 \alpha^2}{\alpha+1} \frac{|z_j|^2}{(1-|z_i|^2)^{\alpha}}. $$
For any compact operator $K: F(p,q,s) \rightarrow \mathcal{Z}_{\mu} $, using the definition of the operator norm and the norm in  Zygmund-type space and \eqref{r12}, we have
\begin{align*}
  \|  L_g - K \|_{F(p,q,s) \rightarrow \mathcal{Z}_{\mu}} &\succeq
  \limsup_{i \rightarrow \infty} \|  L_g h_i - K h_i \|_{\mathcal{Z}_{\mu}} \\
  & \geq
 \limsup_{i \rightarrow \infty}\| L_g h_i \|_{\mathcal{Z}_{\mu}}-
 \limsup_{i \rightarrow \infty}\| K h_i \|_{\mathcal{Z}_{\mu}}\nonumber\\
 & = \limsup_{i \rightarrow \infty} \mu(|z_i|) |\mathcal{R}^2 ( L_g h_i ) (z_i)| \nonumber \\
 & = \limsup_{i \rightarrow \infty} \mu(|z_i|) |\mathcal{R}^2 h_i  (z_i) g(z_i) + \mathcal{R} h_i (z_i) \mathcal{R} g(z_i)|  \nonumber \\
 &=  \limsup_{i \rightarrow \infty} \mu(|z_i|) \left | - \frac{2 \alpha^2}{\alpha+1} \frac{  g(z_i) |z_i|^2}{ (1-|z_i|^2)^{\alpha} } - \frac{2 \alpha^2}{\alpha+1} \frac{  \mathcal{R} g(z_i) |z_i|^2}{ (1-|z_i|^2)^{\alpha} } \right | \\
 & \geq  \limsup_{i \rightarrow \infty}   \frac{2 \alpha^2}{\alpha+1} \frac{   \mu(|z_i|) |\mathcal{R} g(z_i)| |z_i|^2}{ (1-|z_i|^2)^{\alpha} }   - \limsup_{i \rightarrow \infty} \frac{2 \alpha^2}{\alpha+1} \frac{ \mu(|z_i|) |g(z_i)| |z_i|^2}{ (1-|z_i|^2)^{\alpha} } \\
& =  \limsup_{|z| \rightarrow 1}  \frac{ \mu(|z|) |\mathcal{R} g(z)|}{ (1-|z|^2)^{\alpha} }   - \limsup_{|z| \rightarrow 1} \frac{ \mu(|z|) |g(z)| }{ (1-|z|^2)^{\alpha} } \\
& \geq   \limsup_{|z| \rightarrow 1}  \frac{ \mu(|z|) |\mathcal{R} g(z)|}{ (1-|z|^2)^{\alpha} }   - \limsup_{|z| \rightarrow 1} \frac{\mu(|z|)  |g(z)| }{ (1-|z|^2)^{\alpha +1} } \\
& \succeq  \limsup_{|z| \rightarrow 1}  \frac{ \mu(|z|) |\mathcal{R} g(z)|}{ (1-|z|^2)^{\alpha} }   -  \| L_g \|_{e, F(p,q,s) \rightarrow \mathcal{Z}_{\mu}}.
 \end{align*}
Hence
\begin{align*}
\| L_g \|_{e, F(p,q,s) \rightarrow \mathcal{Z}_{\mu}}=  & \inf_{K}\| L_g - K \|_{F(p,q,s) \rightarrow \mathcal{Z}_{\mu}}  \\
 \succeq &  \limsup_{|z| \rightarrow 1}   \frac{  \mu(|z|) |\mathcal{R} g(z)|}{ (1-|z|^2)^{\alpha} }   -  \| L_g \|_{e, F(p,q,s) \rightarrow \mathcal{Z}_{\mu}},
 \end{align*}
and so
\begin{align} \label{r113}
\| L_g \|_{e, F(p,q,s) \rightarrow \mathcal{Z}_{\mu}} \succeq
  \limsup_{|z| \rightarrow 1} \frac{   \mu(|z|) |\mathcal{R} g(z)|}{ (1-|z|^2)^{\alpha} }.
 \end{align}
From \eqref{r12} and \eqref{r113}, the lower estimate of the essential norm is achieved.\\

For the upper estimate, let $\{r_i\}\subset (0,1)$ be a sequence such that $r_i\rightarrow 1$, as $i\rightarrow \infty$. Then as $f_{r_i} \rightarrow f$ uniformly on compact subsets of $\mathbb{B}_n$, as $i \rightarrow \infty$, in which $f_r (z) = f(rz)$. Define the operators $K_r$ on $F(p,q,s)$, $K_r f(z) = f_r(z)$, where $0 < r <1$. Then it is clear that $K_r$ is a compact operator on $F(p,q,s)$. Hence
Hence for any positive integer $i$, the operator $L_g K_{r_i} : F(p,q,s) \rightarrow \mathcal{Z}_{\mu}$ is compact and so
\begin{align} \label{r14}
\|L_g \|_{e, F(p,q,s) \rightarrow \mathcal{Z}_{\mu}} \leq \limsup_{i\rightarrow \infty} \| L_g - L_g K_{r_i}\|_{F(p,q,s) \rightarrow \mathcal{Z}_{\mu}}.
 \end{align}
Therefore for every $f\in F(p,q,s)$, with $\|f\|_{F(p,q,s)}\leq 1$ we have
\begin{align*}
\| L_g f - L_g K_{r_i} f \|_{\mathcal{Z}_{\mu}} = &  \sup_{z \in \mathbb{B}_n} \mu(|z|)
|\mathcal{R}^2 (L_g f )(z) - \mathcal{R}^2 (L_g K_{r_i} f)(z)| \\
= & \sup_{z \in \mathbb{B}_n} \mu(|z|)
|\mathcal{R}^2  f (z) g(z) + \mathcal{R} f(z) \mathcal{R} g(z) \\
&  - r_i \mathcal{R}^2 f_{r_i}(z) g(z) - \mathcal{R} f_{r_i}(z) \mathcal{R} g(z)| \\
\leq &  \sup_{z \in \mathbb{B}_n} \mu(|z|)
|\mathcal{R}^2 f (z) - r_i \mathcal{R}^2 f_{r_i}(z)| | g(z)| \\
& + \sup_{z \in \mathbb{B}_n} \mu(|z|) |\mathcal{R} f (z) - \mathcal{R}f_{r_i}(z) ||\mathcal{R} g(z)| \\
= & \sup_{|z| \leq \delta } \mu(|z|) |\mathcal{R}^2 f (z) - r_i \mathcal{R}^2 f_{r_i}(z)| | g(z)|  \\
& + \sup_{|z| > \delta } \mu(|z|) |\mathcal{R}^2 f (z) - r_i \mathcal{R}^2 f_{r_i}(z)| | g(z)| \\
& + \sup_{|z| \leq \delta } \mu(|z|) |\mathcal{R} f (z) - \mathcal{R}f_{r_i}(z) ||\mathcal{R} g(z)| \\
& + \sup_{|z| > \delta } \mu(|z|) |\mathcal{R} f (z) - \mathcal{R}f_{r_i}(z) ||\mathcal{R} g(z)|,
\end{align*}
where $\delta \in (0,1)$ is fixed. By Weierstrass theorem, $\mathcal{R} f_{r_i} \rightarrow \mathcal{R}f$ and $\mathcal{R}^2 f_{r_i} \rightarrow \mathcal{R}^2 f$ uniformly on compact subsets of $\mathbb{B}_n$, as $i \rightarrow \infty$. Since the ball $\{ z: |z| \leq \delta \}$ is a compact subset of $\mathbb{B}_n$, then we have

\begin{align*}
\limsup_{i\rightarrow \infty} \sup_{|z| \leq \delta} \mu(|z|) |\mathcal{R}^2 f (z) - r_i \mathcal{R}^2 f_{r_i}(z)| | g(z)| = & 0 \\
\limsup_{i\rightarrow \infty} \sup_{|z| \leq \delta} \mu(|z|) |\mathcal{R} f (z) - \mathcal{R}f_{r_i}(z) ||\mathcal{R} g(z)| =& 0.
\end{align*}
Therefore using Lemma \ref{l24} and \eqref{r20} we have
\begin{align*}
\| L_g f - L_g K_{r_i} f \|_{\mathcal{Z}_{\mu}} \leq &
 \sup_{|z| > \delta } \mu(|z|) |\mathcal{R}^2 f (z) - r_i \mathcal{R}^2 f_{r_i}(z)| | g(z)| \\
& + \sup_{|z| > \delta } \mu(|z|) |\mathcal{R} f (z) - \mathcal{R}f_{r_i}(z) ||\mathcal{R} g(z)| \\
\preceq & \sup_{|z| > \delta } \frac{\mu(|z|) |g(z)|}{(1-|z|^2)^{\alpha +1}}+  \sup_{|z| > \delta } \frac{r_i \mu(|z|) |g(z)|}{(1-|r_i z|^2)^{\alpha +1}} \\
& +
\sup_{|z| > \delta } \frac{\mu(|z|) |\mathcal{R}g(z)|}{(1-|z|^2)^{\alpha}} \| f \|_{\mathcal{B}^{\alpha}}
+ \sup_{|z| > \delta } \frac{r_i \mu(|z|) |\mathcal{R}g(z)|}{(1-|r_i z|^2)^{\alpha}} \| f \|_{\mathcal{B}^{\alpha}}
\\
\leq &
\sup_{|z| > \delta } \frac{\mu(|z|) |g(z)|}{(1-|z|^2)^{\alpha +1}}+  \sup_{|z| > \delta } \frac{r_i \mu(|z|) |g(z)|}{(1-|r_i z|^2)^{\alpha +1}} \\
& +
\sup_{|z| > \delta } \frac{\mu(|z|) |\mathcal{R}g(z)|}{(1-|z|^2)^{\alpha}} \| f \|_{F(p,q,s)}
+ \sup_{|z| > \delta } \frac{r_i \mu(|z|) |\mathcal{R}g(z)|}{(1-|r_i z|^2)^{\alpha}} \| f \|_{F(p,q,s)},
\end{align*}
letting $i \rightarrow \infty$, we obtain
\begin{align*}
\| L_g f - L_g K_{r_i} f \|_{\mathcal{Z}_{\mu}} \preceq
2 \sup_{|z| > \delta } \frac{\mu(|z|) |g(z)|}{(1-|z|^2)^{\alpha +1}}
+
2 \sup_{|z| > \delta } \frac{\mu(|z|) |\mathcal{R}g(z)|}{(1-|z|^2)^{\alpha}}.
\end{align*}
When $\delta \rightarrow 1$, from \eqref{r14} we have
\begin{align*}
\|L_g \|_{e, F(p,q,s) \rightarrow \mathcal{Z}_{\mu}} \preceq & \limsup_{|z| \rightarrow 1} \frac{\mu(|z|) |g(z)|}{(1-|z|^2)^{\alpha +1}} +
\limsup_{|z| \rightarrow 1} \frac{\mu(|z|) |\mathcal{R}g(z)|}{(1-|z|^2)^{\alpha}},
 \end{align*}
and consequently we have an upper estimate of the essential norm.
\end{proof}
In the sequel we find an estimate for the essential norm of the Riemann-Stieltjes operator $L_g : H(p,q,\phi) \rightarrow \mathcal{Z}_{\mu}$.
\begin{theorem}
  Let $0 < p,q <\infty$,  $\mu$ be a weight, $\phi$ is a normal function and $g \in H(\mathbb{B}_n)$. If
$L_g :H(p,q,\phi) \rightarrow \mathcal{Z}_{\mu}$ is bounded, then
\begin{equation*}
  \| L_g \|_{e,H(p,q,\phi) \rightarrow \mathcal{Z}_{\mu} } \approx \max
 \left \{ \limsup_{|z| \rightarrow 1} \frac{\mu(|z|) | g(z)|}{\phi(|z|)(1-|z|^2)^{2 + \frac{n}{q}}}, \ \limsup_{|z| \rightarrow 1} \frac{\mu(|z|) | \mathcal{R}g(z)|}{\phi(|z|) (1-|z|^2)^{1+ \frac{n}{q}}} \right \}.
\end{equation*}
\end{theorem}
\begin{proof}
Let $\{ z_i \}_{i \in \mathbb{N}}$ be a sequence in $\mathbb{D}$ such that $|z_i| \rightarrow 1$. Define the sequence of functions as follows
$$ f_i (z) = (\gamma +1)\frac{(1-|z_i|^2)^{b +1}}{\phi(|z_i|)(1-\overline{z_i}z)^{\beta}}  - \gamma \frac{(1-|z_i|^2)^{b+2}}{\phi(|z_j|)(1-\overline{z_i}z)^{\beta +1}}, \ \ \i\in \mathbb{N},$$
where $\gamma = b + 1 + \frac{n}{q}$ and $b$ comes from the definition of the normal function  $\phi$. It is clear that for each $i\in \mathbb{N}$, $f_i \in H(p,q,\phi)$, $\sup_{i} \| f_i \|_{p,q,\phi} \leq C$ and $f_i \rightarrow 0$ uniformly on compact subsets of $\mathbb{B}_n$. Also
$$ \mathcal{R} f_i (z_i) =0, \ \ \ \mathcal{R}^2 f_i (z_i)= -\gamma(\gamma +1) \frac{|z_i|^4}{\phi(|z_i|)(1-|z_i|^2)^{2+ \frac{n}{q}}} \ \ \ \ i\in \mathbb{N}.$$

Hence for any compact operator $K: H(p,q,\phi) \rightarrow \mathcal{Z}_{\mu} $ we have
$ \lim_{n \rightarrow \infty} \| K f_i \|_{\mathcal{Z}_{\mu}} =0$
and so
\begin{align*}
  \|  L_g - K \|_{H(p,q,\phi) \rightarrow \mathcal{Z}_{\mu}} &\succeq
  \limsup_{i \rightarrow \infty} \|  L_g f_i - K f_i \|_{\mathcal{Z}_{\mu}} \\
  &\succeq
 \limsup_{i \rightarrow \infty}\| L_g f_i \|_{\mathcal{Z}_{\mu}}-
 \limsup_{i \rightarrow \infty}\| K f_i \|_{\mathcal{Z}_{\mu}}\nonumber\\
 & \succeq \limsup_{i \rightarrow \infty} \mu(|z_j|) |\mathcal{R}^2 ( L_g f_i ) (z_i)| \nonumber \\
 &\succeq  \limsup_{i \rightarrow \infty} \frac{\mu(|z_j|) | g(z_i)| |z_i|^4}{ \phi(|z_i|)(1-|z_i|^2)^{2+ \frac{n}{q}}}.
 \end{align*}
Therefore we get that
\begin{align} \label{r17}
\| L_g \|_{e, H(p,q,\phi) \rightarrow \mathcal{Z}_{\mu}}= \inf_{K}\| L_g - K \|_{H(p,q,\phi) \rightarrow \mathcal{Z}_{\mu}}  \succeq
  \limsup_{|z| \rightarrow 1} \frac{\mu(|z|) | g(z)| }{ \phi(|z|) (1-|z|^2)^{2 + \frac{n}{q}}}.
 \end{align}
Now we define the sequence $\{ h_j \}$ as follows
$$ h_i (z) = (\beta +2) \frac{(1-|z_i|^2)^{b +1}}{\phi(|z_i|)(1-\overline{z_i}z)^{\beta}}  - \beta \frac{(1-|z_i|^2)^{b+2}}{(1-\overline{z_i}z)^{\beta +1}}.  $$

It is easy to see that $\{ h_i \}$ is a bounded sequence in $H(p,q,\phi)$, which converges to $0$, uniformly on compact subsets of $\mathbb{B}_n$. Also

$$ \mathcal{R} h_i (z_i) =\mathcal{R}^2 h_i (z_i) = \beta   \frac{|z_i|^2}{\phi(|z_i|)(1-|z_i|^2)^{1+\frac{n}{q}}}. $$

Using \eqref{r17} we get that
$$ - \limsup_{|z| \rightarrow 1} \frac{\mu(|z|) | g(z)| }{ \phi(|z|) (1-|z|^2)^{1 + \frac{n}{q}}} \geq - \limsup_{|z| \rightarrow 1} \frac{\mu(|z|) | g(z)| }{ \phi(|z|) (1-|z|^2)^{2 + \frac{n}{q}}}  \succeq - \| L_g \|_{e, H(p,q,\phi) \rightarrow \mathcal{Z}_{\mu}},$$
and then for any compact operator $K: H(p,q,\phi) \rightarrow \mathcal{Z}_{\mu} $
we have
\begin{align*}
  \|  L_g - K \|_{H(p,q,\phi) \rightarrow \mathcal{Z}_{\mu}} &\succeq
  \limsup_{i \rightarrow \infty} \|  L_g h_i - K h_i \|_{\mathcal{Z}_{\mu}} \\
  &\geq
 \limsup_{i \rightarrow \infty}\| L_g h_i \|_{\mathcal{Z}_{\mu}}-
 \limsup_{i \rightarrow \infty}\| K h_i \|_{\mathcal{Z}_{\mu}}\nonumber\\
 & = \limsup_{i \rightarrow \infty} \mu(|z_i|) |\mathcal{R}^2 ( L_g h_i ) (z_j)| \nonumber \\
 & = \limsup_{j \rightarrow \infty} \mu(|z_j|) |\mathcal{R}^2 h_i  (z_i) g(z_j) + \mathcal{R} h_j (z_j) \mathcal{R} g(z_j)|  \nonumber \\
 &=  \limsup_{j \rightarrow \infty} \mu(|z_i|) \left | \gamma  \frac{  g(z_j) |z_j|^2}{ \phi(|z_i|)(1-|z_i|^2)^{1+ \frac{n}{q}} } + \gamma  \frac{  \mathcal{R} g(z_j) |z_j|^2}{  \phi(|z_i|) (1-|z_j|^2)^{1+ \frac{n}{q}} } \right | \\
 & \geq  \limsup_{i \rightarrow \infty}  \gamma \frac{ \mu(|z_i|) |\mathcal{R} g(z_i)| |z_i|^2}{ \phi(|z_i|)(1-|z_i|^2)^{1+ \frac{n}{q}} }   - \limsup_{i \rightarrow \infty} \gamma  \frac{ \mu(|z_i|) |g(z_i)| |z_i|^2}{ \phi(|z_i|) (1-|z_i|^2)^{1+ \frac{n}{q}} } \\
& \succeq  \limsup_{|z| \rightarrow 1}  \frac{ \mu(|z|) |\mathcal{R} g(z)|}{ \phi(|z|)(1-|z|^2)^{1+ \frac{n}{q}} }   - \limsup_{|z| \rightarrow 1} \frac{ \mu(|z|) |g(z)| }{\phi(|z|) (1-|z|^2)^{1+ \frac{n}{q}} } \\
& \succeq  \limsup_{|z| \rightarrow 1}  \frac{ \mu(|z|) |\mathcal{R} g(z)|}{ \phi(|z|) (1-|z|^2)^{1+ \frac{n}{q}} }   -  \| L_g \|_{e, H(p,q,\phi) \rightarrow \mathcal{Z}_{\mu}}.
 \end{align*}
Hence
\begin{align*}
\| L_g \|_{e, H(p,q,\phi) \rightarrow \mathcal{Z}_{\mu}}=  & \inf_{K}\| L_g - K \|_{H(p,q,\phi) \rightarrow \mathcal{Z}_{\mu}}  \\
 \succeq &  \limsup_{|z| \rightarrow 1} \frac{ \mu(|z|) |\mathcal{R} g(z)|}{ \phi(|z|) (1-|z|^2)^{1+ \frac{n}{q}} }   -  \| L_g \|_{e, H(p,q,\phi) \rightarrow \mathcal{Z}_{\mu}},
 \end{align*}
and so
\begin{align} \label{r13}
\| L_g \|_{e, H(p,q,\phi) \rightarrow \mathcal{Z}_{\mu}} \succeq
  \limsup_{|z| \rightarrow 1}  \frac{ \mu(|z|) |\mathcal{R} g(z)|}{ \phi(|z|) (1-|z|^2)^{1+ \frac{n}{q}} }.
 \end{align}
Therefor we have a lower estimate. In order to prove the upper estimate, let $\{r_i\}\subset (0,1)$ be a sequence such that $r_i\rightarrow 1$ as $i\rightarrow \infty$.
Then $f_{r_i} \rightarrow f$ uniformly on compact subsets of $\mathbb{B}_n$, as $i \rightarrow \infty$. We define the operators $T_r$ on $H(p,q,\phi)$, $T_r f(z) = f_r(z)$, where $0 < r <1$. Then $T_r$ is a compact operator on $H(p,q,\phi)$.
So, for any positive integer $i$, the operator $L_g T_{r_i} : H(p,q,\phi) \rightarrow \mathcal{Z}_{\mu}$ is compact. Thus
\begin{align} \label{r18}
\|L_g \|_{e, H(p,q,\phi) \rightarrow \mathcal{Z}_{\mu}} \leq \limsup_{j\rightarrow \infty} \| L_g - L_g T_{r_i}\|_{H(p,q,\phi) \rightarrow \mathcal{Z}_{\mu}}.
 \end{align}
It will be sufficient to compute  $ \| L_g - L_g T_{r_i}\|_{H(p,q,\phi) \rightarrow \mathcal{Z}_{\mu}}$. Let $f\in H(p,q,\phi)$, with $\|f\|_{p,q,\phi}\leq 1$. Then similar to the proof of Theorem \ref{th2} we have
\begin{align*}
\| L_g f - L_g T_{r_i} f \|_{\mathcal{Z}_{\mu}} \leq  & \sup_{|z| \leq \delta } \mu(|z|) |\mathcal{R}^2 f (z) - r_i \mathcal{R}^2 f_{r_j}(z)| | g(z)|  \\
& + \sup_{|z| > \delta } \mu(|z|) |\mathcal{R}^2 f (z) - r_i \mathcal{R}^2 f_{r_i}(z)| | g(z)| \\
& + \sup_{|z| \leq \delta } \mu(|z|) |\mathcal{R} f (z) - \mathcal{R}f_{r_i}(z) ||\mathcal{R} g(z)| \\
& + \sup_{|z| > \delta } \mu(|z|) |\mathcal{R} f (z) - \mathcal{R}f_{r_i}(z) ||\mathcal{R} g(z)|,
\end{align*}
where $\delta \in (0,1)$ is fixed. By Weierstrass theorem, $\mathcal{R} f_{r_i} \rightarrow \mathcal{R}f$ and $\mathcal{R}^2 f_{r_i} \rightarrow \mathcal{R}^2 f$ uniformly on compact subsets of
$\mathbb{B}_n$, as $i \rightarrow \infty$. So
\begin{align*}
\limsup_{i\rightarrow \infty} \sup_{|z| \leq \delta} \mu(|z|) |\mathcal{R}^2 f (z) - r_j \mathcal{R}^2 f_{r_i}(z)| | g(z)| = & 0 \\
\limsup_{i\rightarrow \infty} \sup_{|z| \leq \delta} \mu(|z|) |\mathcal{R} f (z) - \mathcal{R}f_{r_i}(z) ||\mathcal{R} g(z)| =& 0.
\end{align*}
Here we use the fact that $\sup_{z \in \mathbb{B}_n} \mu (|z|) | \mathcal{R} g(z)|<\infty$ and $ \sup_{z \in \mathbb{B}_n} \mu (|z|) |  g(z)|<\infty$ which is obtained from boundedness of the operator.
Using Lemma \ref{l20} we get that
\begin{align*}
\| L_g f - L_g T_{r_i} f \|_{\mathcal{Z}_{\mu}} \leq  &
 \sup_{|z| > \delta } \mu(|z|) |\mathcal{R}^2 f (z) - r_i \mathcal{R}^2 f_{r_i}(z)| | g(z)| \\
& + \sup_{|z| > \delta } \mu(|z|) |\mathcal{R} f (z) - \mathcal{R}f_{r_i}(z) ||\mathcal{R} g(z)| \\
\preceq & \sup_{|z| > \delta } \frac{\mu(|z|) |g(z)| |z|}{\phi(| z|)(1-| z|^2)^{2+ \frac{n}{q}}} \| f \|_{p,q,\phi}+  \sup_{|z| > \delta } \frac{r_i \mu(|z|) |g(z)| |r_i z|}{\phi(|r_i z|)(1-|r_i z|^2)^{2+ \frac{n}{q}}} \| f \|_{p,q,\phi} \\
& +
\sup_{|z| > \delta } \frac{\mu(|z|) |\mathcal{R}g(z)| | z|}{\phi(| z|)(1-|z|^2)^{1 + \frac{n}{q}}} \| f \|_{p,q,\phi}
+ \sup_{|z| > \delta } \frac{r_i \mu(|z|) |\mathcal{R}g(z)||r_i z|}{\phi(|r_i z|)(1-|r_i z|^2)^{1+ \frac{n}{q}}}  \| f \|_{p,q,\phi}
\\
\leq & \sup_{|z| > \delta } \frac{\mu(|z|) |g(z)| |z|}{\phi(| z|)(1-| z|^2)^{2+ \frac{n}{q}}} +  \sup_{|z| > \delta } \frac{r_i \mu(|z|) |g(z)| |r_i z|}{\phi(|r_i z|)(1-|r_i z|^2)^{2+ \frac{n}{q}}}  \\
& +
\sup_{|z| > \delta } \frac{\mu(|z|) |\mathcal{R}g(z)| | z|}{\phi(| z|)(1-|z|^2)^{1 + \frac{n}{q}}}
+ \sup_{|z| > \delta } \frac{r_i \mu(|z|) |\mathcal{R}g(z)||r_i z|}{\phi(|r_i z|)(1-|r_i z|^2)^{1+ \frac{n}{q}}}.
\end{align*}
If $i \rightarrow \infty$, then
\begin{align*}
\| L_g f - L_g T_{r_i} f \|_{\mathcal{Z}_{\mu}} \preceq
2 \sup_{|z| > \delta } \frac{\mu(|z|) |g(z)| |z|}{\phi(| z|)(1-| z|^2)^{2+ \frac{n}{q}}} +
2 \sup_{|z| > \delta } \frac{\mu(|z|) |\mathcal{R}g(z)| | z|}{\phi(| z|)(1-|z|^2)^{1 + \frac{n}{q}}}.
\end{align*}
If $\delta \rightarrow 1$, then \eqref{r18} implies that
\begin{equation*}
\|L_g \|_{e, H(p,q,\phi) \rightarrow \mathcal{Z}_{\mu}}
\preceq
\limsup_{|z| \rightarrow 1}  \frac{\mu(|z|) |g(z)|}{\phi(| z|)(1-| z|^2)^{2+ \frac{n}{q}}} +
\limsup_{|z| \rightarrow 1}  \frac{\mu(|z|) |\mathcal{R}g(z)| }{\phi(| z|)(1-|z|^2)^{1 + \frac{n}{q}}}.
\end{equation*}

\end{proof}

\end{document}